\documentclass{amsart}

\usepackage{amssymb}
\usepackage{latexsym}
\usepackage{amsmath}
\usepackage{euscript}
\usepackage{graphics}
\usepackage[active]{srcltx}

      \def\dR{{\mathbb R}}

   \def\dZ{{\mathbb Z}}

\def\cP{{\mathcal P}}

\def\bm\chi{\mbox{\boldmath$\chi$}}

\def\diag{{\rm diag\,}}

\let\xker=\ker \def\ker{{\xker\,}}

\def\deg{\operatorname{deg}}

\newtheorem{theorem}{Theorem}[section]
\newtheorem{proposition}[theorem]{Proposition}

\newtheorem{definition}[theorem]{Definition}
\theoremstyle{remark}

\newtheorem{remark}[theorem]{Remark}

\numberwithin{equation}{section}

\begin{document}

\title[A note on the GT and Sobolev OPs]{A note on the Geronimus
transformation and Sobolev orthogonal polynomials}

\date{\today}

\author{Maxim Derevyagin}
\address{
Maxim Derevyagin\\
Department of Mathematics MA 4-2\\
Technische Universit\"at Berlin\\
Strasse des 17. Juni 136\\
D-10623 Berlin\\
Germany}
\email{derevyagin.m@gmail.com}

\author{Francisco Marcell\'an}
\address{
Francisco Marcell\'an\\
Departamento de Matem\'aticas\\
Universidad Carlos III de Madrid\\
Avenida de la Universidad 30 \\
28911 Legan\'es\\
Spain}
\email{pacomarc@ing.uc3m.es}


\begin{abstract}
In this note we recast the Geronimus transformation in the framework of polynomials orthogonal with respect to
symmetric bilinear forms. We also show that the double Geronimus transformations lead to
non-diagonal Sobolev type  inner products.
\end{abstract}
\maketitle

\section{Introduction}

Let us consider the following problem. Let $ \{P_{n} \}_{n=0}^{\infty}$ be a sequence of monic
polynomials orthogonal with respect to a nontrivial probability measure supported on an infinite subset of the real line. The problem consists in finding necessary and sufficient conditions
for the real numbers $A_n$, $n=0,1,\dots$, to make the sequence of monic polynomials
\[
P_n(t)+A_nP_{n-1}(t), A_n \neq 0, \quad n=1, 2, \dots,
\]
orthogonal with respect to some measure supported on the real line. The idea of studying this problem goes back to Shohat's paper~\cite{Sh37} concerning quadrature formulas associated to $n$ nodes with a degree of exactness less that $2n-1$.
A few years later after the Shohat's publication a complete and final answer to that problem was given by
Geronimus~\cite{Ger40}.
Thus~\cite{Ger40} provided us with a procedure of constructing new families of orthogonal polynomials from the given ones.
One can also reduce some families of orthogonal polynomials to the known ones with the help
of such a procedure.

Recall that if we have a sequence of monic orthogonal polynomials $\{P_n\}_{n=0}^{\infty}$ then
the polynomial transformation
\[
P_n(t)\to P_n(t)+A_nP_{n-1}(t), A_n\neq 0, \quad n=1,2, \dots,
\]
that gives a new family of orthogonal polynomials, is said to be the Geronimus transformation~\cite{BM04,SpZh95, Zh}.
In fact, the Geronimus transformation divides the measure of orthogonality by the spectral
parameter minus the point of transformation and adds a mass to it at the point of transformation. See also \cite{Maro},
where the sequence of polynomials associated with such a perturbation in a more general algebraic  framework
(orthogonality with respect to a linear functional defined in the linear space of polynomials with complex coefficients)
 is studied.

Besides the measure interpretation, the Geronimus transformation can be also interpreted in terms of Jacobi matrices
in the framework of the so called discrete Darboux transformations and it is related to $LU$ and $UL$ factorizations of
shifted Jacobi matrices~\cite{BM04}. Although the Geronimus transformation has its origin in
mechanical quadrature~\cite{Sh37}, it has also found many applications in classical analysis, numerical analysis, and
physics~\cite{BM04,SpZh95,SpZh97}. In particular, it should be stressed that the Geronimus transformation together
with the Christoffel
transformation (both called discrete Darboux transformations) give a bridge between orthogonal polynomials and discrete
integrable systems~\cite{SpZh95,SpZh97}.

To go deeper in understanding the Geronimus transformation it is somehow natural to consider its iterations.
Say, two iterations of the Geronimus transformation lead to the families of orthogonal polynomials defined by
\[
P_n(t)\to P_n(t)+B_nP_{n-1}(t)+C_nP_{n-2}(t), n\geq 1, C_n\neq 0, n\geq 2.
\]
Such families have been extensively studied in the literature (see \cite{APRM, BM, HHR}, among others).  A particular case of the corresponding inverse problem in terms of perturbations of linear functionals has been analyzed in \cite{BegMaro}.
For more iterations of the Geronimus transformation, see the results contained in \cite{AMPR, KLMP, MaroSfax}.
Some particular cases of inverse problems for the cubic case have been analyzed in \cite{MaroNic}. On the other hand,
in \cite{Il} the higher order ordinary linear differential equations associated with polynomials orthogonal with
respect to iterated Geronimus transformations of Laguerre orthogonal polynomials, the so called Krall-Laguerre
orthogonal polynomials, are studied in a framework of commutative algebras with orthogonal polynomials as eigenfunctions.

An interesting point in analysis of iterations of the Geronimus transformation is the following. It is well known that
the sequence of monic polynomials  $\{\widetilde{Q}^{(\alpha)}_n \}_{n=0}^{\infty}$,
which are orthogonal with respect to the Laguerre-Sobolev type  inner product
\[
[f,g]=\int_{0}^{\infty}f(t)g(t)t^{\alpha}e^{-t}dt+Mf(0)g(0)+Nf'(0)g'(0)\quad f,g\in\cP
\]
defined on the linear space  $\cP$ of polynomials with real coefficients, can be represented in terms of the sequence of classical monic Laguerre polynomials $\{L_n^{(\alpha)}\}_{n=0}^{\infty}$ as follows
\[
\widetilde{Q}^{(\alpha)}_{n}(t)=L_n^{(\alpha+2)}(t)+B_nL_{n-1}^{(\alpha+2)}(t)+C_nL_{n-2}^{(\alpha+2)}(t).
\]
Obviously, one cannot get the Laguerre-Sobolev inner product by dividing  by $t^{2}$ the  measure $ t^{\alpha+2}e^{-t}dt$  and adding masses to it despite the above formula suggests that the Laguerre-Sobolev type  orthogonal polynomials are the two consecutive Geronimus transformations of the classical Laguerre polynomials. This problem brings us to one of the aims of
this note.

One of the  main ideas of the present paper is to include the Laguerre-Sobolev type  orthogonal polynomials and similar Sobolev orthogonal polynomials into the scheme of  Darboux transformations. To this end we propose to reconsider the Geronimus transformation in a more general framework related to symmetric bilinear forms.

Recall that a symmetric bilinear form $B(\cdot,\cdot)$ in the linear space  $\cP$  is a mapping
\[
B(\cdot,\cdot): \cP\times\cP\to\dR
\]
that is linear with respect to each of their arguments and has the symmetry property
\[
B(f,g)=B(g,f),\quad f,g\in\cP.
\]
For instance, the form
\[
(f,g)_0=\int_{\dR} f(t)g(t)d\mu(t)
\]
is symmetric and bilinear. It is not so hard to see that the Gram matrix $\left((t^i,t^j)_0\right)_{i,j=0}^{\infty}$ is
a Hankel matrix and is positive definite.

 A bilinear form is said to be regular (resp. positive definite) if all leading principal submatrices of its Gram matrix
are nonsingular (positive definite). In such cases, the bilinear form generates a sequence of
monic orthogonal polynomials in a simple way by using the Gram-Schmidt process.
Nonetheless, the main advantage of considering bilinear forms in the context of orthogonality
is the ability to include many types of orthogonality such that the corresponding Gram matrix associated with
their moments  is not a Hankel matrix, e.g.
the Sobolev orthogonality (see~\cite{BulBar,BM06}) and other types of orthogonality related to matrix
measures (see~\cite{Dur93}) based on the symmetry of a polynomial operator with respect to a bilinear form.

The paper is organized as follows. In Section~2 the classical Geronimus transformation is considered.
The structure of the symmetric Jacobi matrix corresponding to the transformed polynomials is discussed in the next
section. The double Geronimus transformation in the framework of bilinear forms is presented in Section~4.
The last section gives details of the structure of the symmetric pentadiagonal matrix associated with the recurrence coefficients for the transformed polynomials.

\section{The classical Geronimus transformations}

In this section we review some of the results of~\cite{Ger40} from the point of view of symmetric bilinear forms.

We start with the precise definition of the Geronimus transformation in the framework under consideration.
\begin{definition}\label{SGtrans}Let us consider a symmetric bilinear  form
\[
(f,g)_0=\int_{\dR} f(t)g(t)d\mu(t).
\]
The Geronimus transformation of  $(\cdot,\cdot)_0$ is a symmetric bilinear  form $[\cdot,\cdot]_1$  defined on the set $\cP$ of real polynomials as follows
\begin{equation}\label{defSGT}
[tf(t),g(t)]_1=[f(t),tg(t)]_1=(f,g)_{0}=\int_{\dR} f(t)g(t)d\mu(t), \quad f,g\in{\mathcal P}.
\end{equation}
\end{definition}
Evidently, this definition doesn't determine $[\cdot,\cdot]_1$ uniquely. However, we can see how the Geronimus transformation looks like.
\begin{proposition}\label{SG_h2}
 Suppose that $d\mu$ has the following representation
 \begin{equation}\label{SG_h1}
  d\mu(t)=td\mu_1(t),
 \end{equation}
where $d\mu_1$ is a positive measure and it has finite moments. Then the bilinear form $[\cdot,\cdot]_1$ admits
the representation
\begin{equation}\label{SG_FforBF}
 [f,g]_1=\int_{0}^{\infty}f(t)g(t)d\mu_1(t)+\left(s_0^*-\int_{0}^{\infty}d\mu_1(t)\right)f(0)g(0),\quad f,g\in{\mathcal P},
\end{equation}
where $s_0^*$ is an arbitrary real number.
\end{proposition}
\begin{proof}
  It is clear that the value
  $[1,1]_1$ can be arbitrary. So, let us denote it by $s_0^*$, i.e. $s_0^*=[1,1]_1$.
Further, let us compute $[f,g]_1$ for any $f,g\in{\cP}$:
\[
\begin{split}
[f,g]_1=&[f(t)-f(0)+f(0),g(t)]_1=[f(t)-f(0),g(t)]_1+[f(0),g(t)]_1\\
=&[f(t)-f(0),g(t)]_1+[f(0),g(t)-g(0)]_1+[f(0),g(0)]_1\\
=&\left(\frac{f(t)-f(0)}{t},g(t)\right)+\left(f(0),\frac{g(t)-g(0)}{t}\right)+f(0)g(0)s_0^*\\
=&\int_{0}^{\infty}\frac{f(t)-f(0)}{t}g(t)d\mu(t)+\int_{0}^{\infty}f(0)\frac{g(t)-g(0)}{t}d\mu(t)+f(0)g(0)s_0^*.
\end{split}
\]
Next, using~\eqref{SG_h1} we arrive at
\[
[f,g]_1= \int_{0}^{\infty}\left(f(t)-f(0)\right)g(t)d\mu_1(t)+\int_{0}^{\infty}f(0)\left(g(t)-g(0)\right)d\mu_1(t)+f(0)g(0)s_0^*,
\]
which can be easily simplified to~\eqref{SG_FforBF}.
\end{proof}
To get an idea about the Geronimus transformation, let us consider one particular example of the initial inner product:
\[
 (f,g)_0=\int_0^{+\infty}f(t)g(t)t^\alpha e^{-t}dt,\quad \alpha>0.
\]
Clearly, one of the possible choices for the Geronimus transformation is the following bilinear form
\[
 [f,g]_1=\int_0^{+\infty}f(t)g(t)t^{\alpha-1} e^{-t}dt,\quad \alpha>0,
\]
that is the case where $s_0^*=\int_{0}^{\infty}t^{\alpha-1} e^{-t}dt$.
In this case the forms $(\cdot,\cdot)_0$ and $[\cdot,\cdot]_1$ generate the sequences of monic Laguerre polynomials
$\{L_n ^{(\alpha)}\}_{n=0}^{\infty}$ and $\{L_n ^{(\alpha-1)}\}_{n=0}^{\infty}$, respectively. These polynomials are related as follows
\[
 L_{n} ^{(\alpha)} (t)+ n L_{n-1}^{(\alpha)}(t)=L_{n} ^{(\alpha-1)}(t), \quad n=0,1,\dots.
\]
It turns out that a similar relation is also valid for the Geronimus transformation in general.

\begin{theorem}[cf.~\cite{Ger40}]
Let assume that $(\cdot,\cdot)_0$ and $[\cdot,\cdot]_1$ are positive definite and regular bilinear forms, respectively. Let $\{P_n\}_{n=0}^{\infty}$ be a sequence of monic polynomials
orthogonal with respect to $(\cdot,\cdot)_0$. Then a monic polynomial $P_n^*$ of degree $n$ is orthogonal with respect to $[\cdot,\cdot]_1$
if and only if it can be represented as follows
\begin{equation}\label{SG_hh1}
 P_n^*(t)=\frac{1}{d_n^*}
 \begin{vmatrix}
  P_n(t)&s_0^*P_n(0)+Q_n(0)\\
 P_{n-1}(t) &s_0^*P_{n-1}(0)+Q_{n-1}(0)
 \end{vmatrix},
\end{equation}
where $d_n^*=s_0^*P_{n-1}(0)+Q_{n-1}(0)\ne 0$. Here, $\{Q_n\}_{n=0}^{\infty}$ denotes the sequence of
monic orthogonal polynomials of the second kind with  $\deg Q_{n} = n-1$ and defined by $ Q_{n} (x)=\int_{0}^{\infty} \frac{P_n(t)-P_n(x)}{t-x}d\mu(t)$.
\end{theorem}
\begin{proof}
Since $[\cdot,\cdot]_1$ is regular there exists the corresponding sequence of monic orthogonal polynomials.
Suppose that $P_n^*$ is orthogonal, that is,
 \[
  [P_n^*(t),t^k]_1=[t^k,P_n^*(t)]_1=0,\quad k=0,1,2,\dots, n-1.
 \]
In turn, for the original bilinear form we have
\[
 (P_n^*(t),t^{k-1})_0=[P_n^*(t),t^k]_1=0,\quad k=0,1,2,\dots, n-1,
\]
which obviously implies that
\begin{equation}\label{SD_OP_repr}
 P_n^*(t)=P_n(t)+A_n P_{n-1}(t),
\end{equation}
where $A_n$ is a real number. Next, let us calculate $A_n$, $n\geq1$. To this end
we are going to use the following equation
\[
 \begin{split}
  0=&[P_n^*(t),1]=s_0^*P_n^*(0)+\left(\frac{P_n^*(t)-P_n^*(0)}{t},1\right)=\\
  =&s_0^*P_n^*(0)+\left(\frac{P_n(t)-P_n(0)}{t},1\right)+A_n\left(\frac{P_{n-1}(t)-P_{n-1}(0)}{t},1\right)\\
  =&s_0^*P_n^*(0)+\int_0^{\infty}\frac{P_n(t)-P_n(0)}{t}d\mu(t)+
A_n\int_0^{\infty}\frac{P_{n-1}(t)-P_{n-1}(0)}{t}d\mu(t)\\
  =&s_0^*(P_n(0)+A_nP_{n-1}(0))+Q_n(0)+A_nQ_{n-1}(0)\\
  =&s_0^*P_n(0)+Q_n(0)+A_n(s_0^*P_{n-1}(0)+Q_{n-1}(0)), n\geq1.
 \end{split}
\]
We see that the equation is equivalent to the orthogonality of the polynomial
$P_n+A_n P_{n-1}$ with respect to $[\cdot,\cdot]_1$. Hence, the equation
has a unique solution and, so, one has
\[
 s_0^*P_{n-1}(0)+Q_{n-1}(0)\ne 0.
\]
Furthermore, the unique solution of the above equation  is
\begin{equation}
 A_n=-\frac{s_0^*P_{n}(0)+Q_{n}(0)}{s_0^*P_{n-1}(0)+Q_{n-1}(0)},
\end{equation}
which leads us to formula~\eqref{SG_hh1}.
\end{proof}
Finally, it is worth mentioning that the Geronimus transformation can be also
considered in the case when $(\cdot,\cdot)_0$ is regular and is not necessarily positive definite~\cite{BM04}.
Moreover, necessary and sufficient conditions for the regularity of $[\cdot, \cdot]_{1}$ are analyzed in ~\cite{BM04, D13, DD11}.

\section{The structure of the transformed Jacobi matrix}

It is very well known that, assuming $[\cdot, \cdot]_{1}$ is positive definite, we can  associate with the sequence of monic orthogonal polynomials $\{P_n^{*}\}_{n=0}^{\infty}$ a monic tridiagonal Jacobi matrix
\begin{equation*}
J_{mon}^*=\begin{pmatrix}
{b}_{0}^*   & 1 &       &\\
({c}_{0}^*)^2   &{b}_1^*    &{1}&\\
        &({c}_1^*)^2    &{b}_{2}^* &\ddots\\
&       &\ddots &\ddots\\
\end{pmatrix}.
\end{equation*}
Recall that the entries of $J_{mon}^*$ are defined by the corresponding three-term recurrence
relation
\begin{equation}\label{recrelmonp}
t P_j^*(t) = P_{j+1}^*(t) + b_j^*P_j^*(t) + (c_{j-1}^*)^2P_{j-1}^*(t),\quad j\in\dZ_+,
\end{equation}
with the initial conditions
\[
P_{-1}^*(t)=0,\quad P_{0}^*(t) =1,
\]
where $b_j^*\in\dR$ and $c_j^*>0$, $j\in\dZ_+$.

Depending on circumstances it can be also convenient to consider a symmetric tridiagonal Jacobi matrix
\[
 J^*=\begin{pmatrix}
{b}_{0}^*   & {c}_{0}^* &       &\\
{c}_{0}^*   &{b}_1^*    &{c}_{0}^*&\\
        &{c}_1^*    &{b}_{2}^* &\ddots\\
&       &\ddots &\ddots\\
\end{pmatrix}=\Psi^{-1}J_{mon}^*\Psi,
\]
where $\Psi=\diag(1,c_0^*,c_0^*c_1^*,c_0^*c_1^*c_2^*,\dots)$. Indeed,   $J^*$ is the  matrix of
the multiplication operator with respect to the basis of orthonormal polynomials
\[
 \widehat{P}_n^*(t)=\frac{1}{h_n^*}P_n^*(t),\quad (h_n^*)^2=[P_n^*,P_n^*]_1,\quad h_n^*>0.
\]
In other words, we have the following representation
\[
 J^*=\left([t\widehat{P}_n^*(t),\widehat{P}_m^*(t)]_1\right)_{n,m=0}^{\infty}.
\]
Since $J^*$ corresponds to the Geronimus transformation, it has a special structure, which can be expressed
in terms of the coefficients $A_{n}$, $n= 1,2,\dots$, and the free parameter.
\begin{theorem}[cf.~\cite{BM04}]
Let us assume that $(\cdot,\cdot)_0$, $[\cdot,\cdot]_1$ are positive definite
and $\{P_n\}_{n=0}^{\infty}$ and $\{P_n^{*}\}_{n=0}^{\infty}$ are, respectively,  the corresponding sequences of monic orthogonal polynomials.
 Then the matrix $J^*$ admits the following Cholesky decomposition
 \begin{equation}\label{SD_Cholesky}
  J^*=LL^{\top},
 \end{equation}
where the bidiagonal lower triangular matrix $L$ has the form
\begin{equation}\label{SD_bidiag}
 L=\begin{pmatrix}
\frac{h_0}{s_0^*}& 0&       &&\\
A_1\frac{h_0}{s_0^*} & \frac{h_1}{\sqrt{A_1}h_0}   & 0 &&\\
        &\frac{A_2h_1}{\sqrt{A_1}h_0}   & \frac{h_2}{\sqrt{A_2}h_1} &0&\\
     &   &\frac{A_3h_2}{\sqrt{A_2}h_1}   & \frac{h_3}{\sqrt{A_3}h_2} &\ddots\\
&&       &\ddots &\ddots\\
\end{pmatrix}.
\end{equation}
\end{theorem}
\begin{proof}
We begin by noticing that
\begin{equation}\label{SD_mrepr_h1}
J^*=\begin{pmatrix}
\frac{1}{h_0^*}  & 0 &  \\
0   &\frac{1}{h_1^*}   &\ddots\\
     &\ddots &\ddots\\
\end{pmatrix}
\begin{pmatrix}
[tP_0^*(t),P_0^*(t)]_1  & [tP_0^*(t),P_1^*(t)]_1 &  \\
[tP_1^*(t),P_0^*(t)]_1   &[tP_1^*(t),P_1^*(t)]_1   &\ddots\\
     &\ddots &\ddots\\
\end{pmatrix}
\begin{pmatrix}
\frac{1}{h_0^*}  & 0 &  \\
0   &\frac{1}{h_1^*}   &\ddots\\
     &\ddots &\ddots\\
\end{pmatrix}.
\end{equation}
Since $[tP_n^*(t),P_m^*(t)]_1=(P_n^*(t),P_m^*(t))_0$, the symmetric tridiagonal matrix in the middle of the right hand side of~\eqref{SD_mrepr_h1}
reduces to
\[
\begin{split}
 \left((P_n^*,P_m^*)_0\right)_{n,m=0}^{\infty}=&\begin{pmatrix}
(P_0^*,P_0^*)_0  & (P_0^*,P_1^*)_0 &  0&\\
(P_1^*,P_0^*)_0   &(P_1^*,P_1^*)_0   &(P_1^*,P_2^*)_0&\\
      0  &(P_2^*,P_1^*)_0   &(P_2^*,P_2^*)_0 &\ddots\\
&       &\ddots &\ddots\\
\end{pmatrix}\\
=&\begin{pmatrix}
h_0^2  & A_1h_0^2 &  0&\\
A_1h_0^2  &h_1^2+A_1^2h_0^2   &A_2h_1^2&\\
      0  &A_2h_1^2   &h_2^2+A_2^2h_1^2&\ddots\\
&       &\ddots &\ddots\\
\end{pmatrix}\\
=&\begin{pmatrix}
h_0   & 0 &       &\\
{A}_{1}h_0   & h_1    &0 &\\
        &{A}_2h_1    &h_2 &\ddots\\
&       &\ddots &\ddots\\
\end{pmatrix}\begin{pmatrix}
{h}_{0}   & A_1h_0 &       &\\
{0}   &{h}_1    &{A}_2h_1&\\
        &0    &{h}_{2} &\ddots\\
&       &\ddots &\ddots\\
\end{pmatrix}
\end{split}
\]
in view of formula~\eqref{SD_OP_repr}.
Hence it is clear that~\eqref{SD_Cholesky} holds with
\begin{equation}\label{SD_mrepr_h2}
L=\begin{pmatrix}
\frac{1}{h_0^*}  & 0 &  \\
0   &\frac{1}{h_1^*}   &\ddots\\
     &\ddots &\ddots\\
\end{pmatrix}\begin{pmatrix}
h_0   & 0 &       &\\
{A}_{0}h_1   & h_1    &0 &\\
        &{A}_2h_1    &h_2 &\ddots\\
&       &\ddots &\ddots\\
\end{pmatrix}.
\end{equation}
Now, observe that
\[
\begin{split}
(h_{n+1}^*)^2=&[P_{n+1}^*,P_{n+1}^*]_1=[tP_{n}^*(t),P_{n+1}^*(t)]_1=(P_{n}^*,P_{n+1}^*)_0\\
=&(P_n+A_nP_{n-1},P_{n+1}+A_{n+1}P_n)_0\\
=&A_{n+1}h_n^2,
\end{split}
\]
which gives
\[
h_{n+1}^*=\sqrt{A_{n+1}}h_n^*.
\]
Combining this and $(h_{0}^*)^2=s_0^*$ we get that~\eqref{SD_mrepr_h2} can be easily simplified to~\eqref{SD_bidiag}.
\end{proof}

As a matter of fact, this statement is a trace of the fact that the Geronimus transformation can be interpreted in the matrix language (for details, see~\cite{BM04},
as well as ~\cite{DD11} for the non-regular case).

In order the paper to be self-contained a direct connection between the matrices   $J_{mon}$ and  $J_{mon}^{*}$ associated with the monic orthogonal polynomial sequences $\{P_n\}_{n=0}^{\infty}$, $\{P_n^{*}\}_{n=0}^{\infty}$, respectively, will be stated.
Let
\begin{equation}
 L_{mon}= \begin{pmatrix}
1 & 0&       &&\\
A_{1} & 1   & 0 &&\\
0  &A_2   & 1 &0&\\
   0  & 0  &A_3 &1 &\ddots\\
&&       &\ddots &\ddots\\
\end{pmatrix}
\end{equation}
be an infinite matrix such that   $ P^{*} = L_{mon} P$, where  $P^{*}= (P_{0}^{*}, P_{1}^{*}, \dots )^{\top}$  and
$P= (P_{0}, P_{1}, \dots )^{\top}.$ On the other hand, according to the Christoffel formula (see \cite{Chi})
or, equivalently, to the fact that
\[
[tP_n(t),P_m^*(t)]_1=(P_n(t),P_m^*(t))_0,\quad m=0,\dots, n-1,
\]
we get the relation
\[
tP_{n}(t)= P_{n+1}^*(t) + F_{n+1} P_{n}^*(t),\quad  F_{n+1}\neq0,\quad n\geq 0.
\]
In matrix terms,  we have that $t P = U_{mon}P^{*} $, where
\begin{equation}
U_{mon} ^{\top} = \begin{pmatrix}
F_1 & 0&       &&\\
1 & F_2   & 0 &&\\
0  &1   & F_3 &0&\\
   0  & 0 &1 &F_4 &\ddots\\
&&       &\ddots &\ddots\\
\end{pmatrix}
\end{equation}

Thus, we can state the following
\begin{theorem} We have that
 \begin{equation}
 J_{mon} = U_{mon} L_{mon},
\end{equation}
 \begin{equation}
 J_{mon}^{*}=  L_{mon }U_{mon}.
\end{equation}
\end{theorem}

\begin{proof}

Notice that
\begin{equation}
  t  P=   U_{mon} P^{*}= U_{mon} L_{mon} P.
\end{equation}  Thus, one gets
\begin{equation}
 J_{mon} = U_{mon} L_{mon}.
\end{equation}
On the other hand, we se that
\begin{equation}
   t P^{*} = L_{mon} t P =  L_{mon} U_{mon} P^{*}.
\end{equation}
 As a consequence, we arrive at
\begin{equation}
  J_{mon}^{*}=  L_{mon } U_{mon}.
\end{equation}
\end{proof}

Thus, we have a simple proof of a very well known result  (see \cite{BM04}) in terms of a Darboux transformation
with parameter (see also~\cite{D13,DD11} for similar results in the non-regular case).

\section{The double Geronimus transformation and the Sobolev orthogonality}

In this section we present the double Geronimus transformation in the framework of symmetric bilinear forms.
Also, it is shown that this transformation leads to Sobolev inner products
and, therefore, to Sobolev orthogonal polynomials.

First, let us clarify what we mean by the double Geronimus transformation.
\begin{definition}\label{DGtrans}Let us consider a symmetric bilinear  form
$(\cdot,\cdot)_0$. The double Geronimus transformation of  $(\cdot,\cdot)_0$ is a symmetric bilinear form defined on the linear space $\mathcal P$ of polynomials with real coefficients  as follows
\begin{equation}\label{defGT}
[t^2f(t),g(t)]_2=[f(t),t^2g(t)]_2=(f,g)_{0}=\int_{\dR} f(t)g(t)d\mu(t), \quad f,g\in{\mathcal P}.
\end{equation}
\end{definition}

From~\eqref{defGT} one can see that the form $[\cdot,\cdot]_2$ is not uniquely defined. In particular, the symmetric matrix (since the form is symmetric)
\[
\begin{pmatrix}
 [1,1] &[1,t] \\
 [t,1]& [t,t]\\
\end{pmatrix}=
\begin{pmatrix}
 s_0^{**} &s_1^{**} \\
 s_1^{**}& s_2^{**}\\
\end{pmatrix}
\]
can be chosen arbitrarily. Despite this, one can see the structure of the double Geronimus transformation.

\begin{proposition}
 Suppose that $d\mu$ has the following representation
 \begin{equation}\label{DG_h1}
  d\mu(t)=t^2d\mu_2(t),
 \end{equation}
where $d\mu_2$ is a positive measure and it has finite moments. Then the double Geronimus transformation of $(\cdot,\cdot)_0$ admits
the representation
\begin{equation}\label{DG_FforBF}
 [f,g]_2=\int_{\dR}f(t)g(t)d\mu_2(t)+
 \begin{pmatrix} f(0)& f'(0)\\ \end{pmatrix}
M\begin{pmatrix} g(0)\\ g'(0)\\ \end{pmatrix},
\end{equation}
where the symmetric  matrix $M$ has the following form
\[
 M=\begin{pmatrix}
 s_0^{**} &s_1^{**} \\
 s_1^{**}& s_2^{**}\\
\end{pmatrix}-
 \begin{pmatrix}
 \int_{\dR}d\mu_2(t) &\int_{\dR}td\mu_2(t) \\
 \int_{\dR}td\mu_2(t)& \int_{\dR}t^2d\mu_2(t)\\
\end{pmatrix} .
\]
\end{proposition}
\begin{proof}
The proof is similar to that of Proposition~\ref{SG_h2}. First, we see that
\[
\begin{split}
 [f,g]_2=&[f(t)-f(0)-tf'(0),g(t)]_2+[f(0)+tf'(0),g(t)]_2\\
 =&\left(\frac{f(t)-f(0)-tf'(0)}{t^2},g(t)\right)+[f(0)+tf'(0),g(t)]_2.
 \end{split}
\]
Then, making use of the representation  $g(t)=g(t)-g(0)-tg'(0)+g(0)+tg'(0)$ and
taking into account~\eqref{DG_h1} we get the desired result~\eqref{DG_FforBF}.
\end{proof}

Since the symmetric matrix $M$ can be arbitrary,
from formula~\eqref{DG_FforBF} one can see that, in general, the double Geronimus transformation $[\cdot,\cdot]_2$
generates Sobolev type  inner products. In particular,  one recovers the positive diagonal Sobolev type inner products when
\[
M=\begin{pmatrix}
 \lambda_1 &0 \\
 0& \lambda_2\\
\end{pmatrix}, \quad \lambda_1\ge 0\text{ and } \lambda_2> 0.
\]
 At the same time, we also get that the double Geronimus transformation adds a matrix mass at the point in some sense.
Thus, the corresponding sequence of orthogonal polynomials, which are called in the literature  Sobolev type  orthogonal
polynomials, are not already standard scalar orthogonal  polynomials (the scalar Hankel structure of the Gram matrix
is destroyed by the perturbation)
but not yet essentially matrix orthogonal polynomials although it is convenient to consider them as matrix orthogonal
since the Gram matrix is in fact a $2\times 2$ block Hankel matrix
(see~\cite{AMRR} for some basic properties
of Sobolev type orthogonal polynomials).

Now we are in a position to give an explicit formula for the transformed polynomials $\{P_n^{**}\}_{n=0}^{\infty}$
orthogonal with respect to $[\cdot,\cdot]_2$ in terms of the original polynomials $\{P_n\}_{n=0}^{\infty}$.
\begin{theorem}
Let us assume that $(\cdot,\cdot)_0$ and $[\cdot,\cdot]_2$ are both positive definite and regular bilinear forms, respectively. Let $\{P_n\}_{n=0}^{\infty}$ be a sequence of monic polynomials
orthogonal with respect to $(\cdot,\cdot)_0$. Then a monic polynomial $P_n^{**}$ of degree $n$ is orthogonal with respect to $[\cdot,\cdot]_2$
if and only if it can be represented as follows
\footnotesize
\begin{equation}\label{DG_PolRepr}
 P_n^{**}(t)=\frac{1}{d_n^{**}}
 \begin{vmatrix}
  P_n(t)&R'_n(0;s_1^{**})+s_0^{**}P_n(0)&R_{n}(0;s_1^{**})+(s_2^{**}-s_0^{**})P'_{n}(0)\\
P_{n-1}(t) &R'_{n-1}(0;s_1^{**})+s_0^{**}P_{n-1}(0) & R_{n-1}(0;s_1^{**})+(s_2^{**}-s_0^{**})P'_{n-1}(0)\\
P_{n-2}(t)& R'_{n-2}(0;s_1^{**})+s_0^{**}P_{n-2}(0)&  R_{n-2}(0;s_1^{**})+(s_2^{**}-s_0^{**})P'_{n-2}(0)
 \end{vmatrix},
\end{equation}
\normalsize
where $R_n(t;s)=sP_{n}(t)+Q_{n}(t)$, $R'_n(t;s)=sP'_{n}(t)+Q'_{n}(t)$, and
\[
d_n^{**}=
 \begin{vmatrix}
R'_{n-1}(0;s_1^{**})+s_0^{**}P_{n-1}(0) & R_{n-1}(0;s_1^{**})+(s_2^{**}-s_0^{**})P'_{n-1}(0)\\
 R'_{n-2}(0;s_1^{**})+s_0^{**}P_{n-2}(0)&  R_{n-2}(0;s_1^{**})+(s_2^{**}-s_0^{**})P'_{n-2}(0)
 \end{vmatrix}
\]
is nonzero.
\end{theorem}
\begin{proof}
 The orthogonality of $P_n^{**}$ is equivalent to the following condition
\[
[P_n^{**}(t),t^k]_2=0, \quad k=0,\dots,n-1,
\]
which for  $n\ge 3$ further reduces to
\[
(P_n^{**}(t),t^{k-2})_{0}=0, \quad k=2,\dots, n-1.
\]
The latter relation is obviously equivalent to the representation
\begin{equation}\label{DD_OP_repr}
 P_n^{**}(t)=P_n(t)+B_nP_{n-1}(t)+C_nP_{n-2}(t).
\end{equation}
Therefore one can see that the coefficients $B_n$ and $C_n$ are uniquely determined by the relations
\[
 \begin{split}
  [P_n^{**}(t),1]_2=&0,\\
  [P_n^{**}(t),t]_2=&0,
 \end{split}
\]
which can be rewritten as follows
\begin{equation}\label{DG_h2}
 \begin{split}
  [P_n(t),1]_2+B_n[P_{n-1}(t),1]_2+C_n[P_{n-2}(t),1]_2=& 0,\\
  [P_n(t),t]_2+B_n[P_{n-1}(t),t]_2+C_n[P_{n-2}(t),t]_2=& 0.
 \end{split}
\end{equation}
Since the monic orthogonal polynomial $P_n^{**}$ of degree $n$ is uniquely defined, the system~\eqref{DG_h2}
has a unique solution. Indeed, on the one hand, it is clear that the system has at least one solution because there exists a monic
orthogonal polynomial of degree $n$. On the other hand, if it has two different solutions then these solutions would give two different
monic orthogonal polynomials of degree $n$. The latter fact is not possible according the uniqueness of such a sequence. So, we conclude that
the determinant
\[
 d_n^{**}=\begin{vmatrix}
 [P_{n-1}(t),1]_2 &[P_{n-1}(t),t]_2\\
[P_{n-2}(t),1]_2 & [P_{n-2}(t),t]_2
 \end{vmatrix}
\]
is nonzero and  the orthogonality of $P_n^{**}$ is
equivalent to the representation
\[
 P_n^{**}(t)=\frac{1}{d_n^{**}}
 \begin{vmatrix}
  P_n(t)&[P_n(t),1]_2&[P_n(t),t]_2\\
P_{n-1}(t) &[P_{n-1}(t),1]_2 &[P_{n-1}(t),t]_2\\
P_{n-2}(t)&[P_{n-2}(t),1]_2 & [P_{n-2}(t),t]_2
 \end{vmatrix}.
\]
Now, to get formula~\eqref{DG_PolRepr} it remains to re-express the entries of the corresponding determinants. To this end,
for the last column one can get
\[
\begin{split}
 [P_n(t),t]_2=&[P_n(t)-P_n(0)-tP'_n(0),t]_2+P_n(0)[1,t]_2+P'_n(0)[t,t]_2\\
 =&\left(\frac{P_n(t)-P_n(0)-tP'_n(0)}{t^2},t\right)_0+s_1^{**}P_n(0)+s_2^{**}P'_n(0)\\
 =&\left(\frac{P_n(t)-P_n(0)-tP'_n(0)}{t},1\right)_0+s_1^{**}P_n(0)+s_2^{**}P'_n(0)\\
 =&Q_n(0)+s_1^{**}P_n(0)+(s_2^{**}-s_0^{**})P'_n(0)\\
 =& R_n(0;s_1^{**})+(s_2^{**}-s_0^{**})P'_n(0).
\end{split}
\]
Next, we also have that
\[
 \begin{split}
[P_n(t),1]_2=&[P_n(t)-P_n(0)-tP'_n(0),1]_2+P_n(0)[1,1]_2+P'_n(0)[t,1]_2\\
=&\left(\frac{P_n(t)-P_n(0)-tP'_n(0)}{t^2},1\right)_0+s_0^{**}P_n(0)+s_1^{**}P'_n(0).\\
 \end{split}
\]
On the other hand, notice that
\[
\begin{split}
 Q'_n(0)=&\lim_{\epsilon\to 0}\frac{Q_n(\epsilon)-Q_n(0)}{\epsilon}\\
 =&\lim_{\epsilon\to 0}\frac{1}{\epsilon}\left(\left(\frac{P_n(t)-P_n(\epsilon)}{t-\epsilon},1\right)_0
 -\left(\frac{P_n(t)-P_n(0)}{t},1\right)_0\right)\\
 =&\lim_{\epsilon\to 0}\frac{1}{\epsilon}\int_{\dR}\frac{\epsilon(P_n(t)-P_n(0))-t(P_n(\epsilon)-P_n(0))}{t(t-\epsilon)}d\mu(t)\\
 =&\int_{\dR}\frac{P_n(t)-P_n(0)-tP'_n(0)}{t^2}d\mu(t)\\
 =&\left(\frac{P_n(t)-P_n(0)-tP'_n(0)}{t^2},1\right)_0.
 \end{split}
\]
Note that we can interchange the limit and integral due to Lebesgue's dominated convergence theorem. Finally,  we  get
\[
 \begin{split}
[P_n(t),1]_2=& Q'_n(0)+s_0^{**}P_n(0)+s_1^{**}P'_n(0)\\
=& R'_n(0;s_1^{**})+s_0^{**}P_n(0),
 \end{split}
\]
which completes the proof.
\end{proof}

\section{The structure of the transformed pentadiagonal matrix}

In the case when the sequence of  polynomials $\{P_n^{**}\}_{n=0}^{\infty}$ is orthogonal with respect
to an inner product of the form~\eqref{DG_FforBF},
it is quite natural to consider the matrix representation of the square of the multiplication operator~\cite{Dur93,ELMMR}.
Indeed, according to formula~\eqref{DG_FforBF} the multiplication operator is not necessarily symmetric with
respect to $[\cdot,\cdot]_2$ and as a consequence one cannot apply some classical tricks in this case. Nevertheless,
the square of the multiplication operator is symmetric by the definition of the double Geronimus transformation.
Thus, the classical machinery works for this symmetric operator. So, assuming $[\cdot, \cdot]_{2}$ is positive definite, let us introduce the following symmetric matrix
\[
 J^{**}=\left([t^2\widehat{P}_n^{**}(t),\widehat{P}_m^{**}(t)]_2\right)_{n,m=0}^{\infty},
\]
where the corresponding orthonormal polynomials $\{\widehat{P}_n^{**}\}_{n=0}^{\infty}$ are given by
\[
 \widehat{P}_n^{**}(t)=\frac{1}{h_n^{**}}P_n^{**}(t),\quad (h_n^{**})^2=[P_n^{**},P_n^{**}]_2,\quad h_n^{**}>0.
\]
In fact, $J^{**}$ is pentadiagonal and the following statement holds true.
\begin{theorem}
Let us assume that $(\cdot,\cdot)_0$, $[\cdot,\cdot]_2$ are positive definite
and  $\{P_n\}_{n=0}^{\infty}$, $\{P_n^{**}\}_{n=0}^{\infty}$, respectively,  are the corresponding sequences of monic orthogonal polynomials.
Then the matrix $J^{**}$ admits the following Cholesky decomposition
 \begin{equation}\label{DD_Cholesky}
  J^{**}=LL^{\top},
 \end{equation}
where the lower triangular matrix $L$ has only three nonvanishing  diagonals and is of the form
\begin{equation}\label{DD_bidiag}
 L=\begin{pmatrix}
\frac{h_0}{h_0^{**}}& 0&       &&\\
B_1\frac{h_0}{h_0^{**}} & \frac{h_1}{h_1^{**}}   & 0 &&\\
C_2\frac{h_0}{h_0^{**}}  &B_2\frac{h_1}{h_1^{**}}   &\frac{h_2}{h_2^{**}} &0&\\
   0  & C_3\frac{h_1}{h_1^{**}}  &B_3\frac{h_2}{h_2^{**}}  &\frac{h_3}{h_3^{**}} &\ddots\\
&&       &\ddots &\ddots\\
\end{pmatrix},
\end{equation}
where the ratio $h_{n+2}/h_{n+2}^{**}$ can be expressed in terms of the coefficients $B_{n}$  and $C_{n}$, $n=1, 2,...$ of the linear combination  as follows
\begin{equation}\label{DG_mrepr_h1}
\frac{h_{n+2}}{h_{n+2}^{**}}=\frac{h_{n+2}}{\sqrt{C_{n+2}}h_n},\quad n=0,1,\dots.
\end{equation}
\end{theorem}
\begin{remark}
It should be stressed that $h_0^{**}$ and $h_1^{**}$ can be parametrized by the free parameters:
\[
(h_0^{**})^2=s_0^{**},\quad (h_1^{**})^2=s_2^{**}+s_1^{**}\left(B_1-\frac{s_1}{s_0}\right).
\]
\end{remark}
\begin{proof}
Obviously, we have that
\footnotesize
\begin{equation}\label{DD_mrepr_h1}
J^{**}=\begin{pmatrix}
\frac{1}{h_0^{**}}  & 0 &  \\
0   &\frac{1}{h_1^{**}}   &\ddots\\
     &\ddots &\ddots\\
\end{pmatrix}
\begin{pmatrix}
[t^2P_0^{**}(t),P_0^{**}(t)]_2  & [t^2P_0^{**}(t),P_1^{**}(t)]_2 &  \\
[t^2P_1^{**}(t),P_0^{**}(t)]_2   &[t^2P_1^{**}(t),P_1^{**}(t)]_2   &\ddots\\
     &\ddots &\ddots\\
\end{pmatrix}
\begin{pmatrix}
\frac{1}{h_0^{**}}  & 0 &  \\
0   &\frac{1}{h_1^{**}}   &\ddots\\
     &\ddots &\ddots\\
\end{pmatrix}.
\end{equation}
\normalsize
Since $[t^2P_n^{**}(t),P_m^{**}(t)]_2=(P_n^*(t),P_m^*(t))_0$, the pentadiagonal matrix in the middle of the right
hand side of~\eqref{DD_mrepr_h1} reduces to
\footnotesize
\[
\begin{split}
 \left((P_n^{**},P_m^{**})_0\right)_{n,m=0}^{\infty}=&\begin{pmatrix}
(P_0^{**},P_0^{**})_0  & (P_0^{**},P_1^{**})_0 &  (P_0^{**},P_2^{**})_0&0&\\
(P_1^{**},P_0^{**})_0   &(P_1^{**},P_1^{**})_0   &(P_1^{**},P_2^{**})_0&(P_1^{**},P_3^{**})_0&\\
      (P_2^{**},P_0^{**})_0  &(P_2^{**},P_1^{**})_0   &(P_2^{**},P_2^{**})_0 &(P_2^{**},P_3^{**})_0&\ddots\\
      0&(P_3^{**},P_1^{**})_0  &(P_3^{**},P_2^{**})_0   &(P_3^{**},P_3^{**})_0 &\ddots\\
&    & \ddots  &\ddots &\ddots\\
\end{pmatrix}\\
=&\begin{pmatrix}
h_0^2 & B_1h_0^2 &  C_2h_0^2&0&\\
B_1h_0^2   &h_1^2+B_1^2h_0^2&B_2h_1^2+B_1C_2h_0^2&C_3h_1^2&\\
      C_2h_0^2  &B_2h_1^2+B_1C_2h_0^2  &h_2^2+B_2^2h_1^2+C_2^2h_0^2&B_3h_2^2+B_2C_3h_1^2&\ddots\\
      0&C_3h_1^2  &B_3h_2^2+B_2C_3h_1^2  &h_3^2+B_3^2h_2^2+C_3^2h_1^2 &\ddots\\
&    & \ddots  &\ddots &\ddots\\
\end{pmatrix}\\
=&\begin{pmatrix}
h_0   & 0 &       &&\\
{B}_{1}h_0   & h_1    &0 &&\\
  C_2h_0      &{B}_2h_1    &h_2 &0&\\
0&C_3h_1      &{B}_3h_2    &h_3 &\ddots\\
&      & \ddots&\ddots &\ddots\\
\end{pmatrix}\begin{pmatrix}
{h}_{0}   & B_1h_0 & C_2h_0      &&\\
{0}   &{h}_1    &{B}_2h_1&C_3h_1&\\
        &0    &{h}_{2} &B_3h_2&\ddots\\
        &&0&h_3&\ddots\\
&&       &\ddots &\ddots\\
\end{pmatrix}
\end{split}
\]
\normalsize
in view of formula~\eqref{DD_OP_repr}.
Hence it is clear that~\eqref{DD_Cholesky} holds with
\begin{equation}\label{DD_mrepr_h2}
L=\begin{pmatrix}
\frac{1}{h_0^{**}}  & 0 &  \\
0   &\frac{1}{h_1^{**}}   &\ddots\\
     &\ddots &\ddots\\
\end{pmatrix}\begin{pmatrix}
h_0   & 0 &       &&\\
{B}_{1}h_0   & h_1    &0 &&\\
  C_2h_0      &{B}_2h_1    &h_2 &0&\\
0&C_3h_1      &{B}_3h_2    &h_3 &\ddots\\
&      & \ddots&\ddots &\ddots\\
\end{pmatrix}.
\end{equation}
Thus we arrive at~\eqref{DD_bidiag} after simple computations.
Finally, it remains to see that
\[
\begin{split}
(h_{n+1}^{**})^2=&[P_{n+2}^{**},P_{n+2}^{**}]_2=[t^2P_{n}^{**}(t),P_{n+2}^*(t)]_2=(P_{n}^{**},P_{n+2}^{**})_0\\
=&(P_n+B_nP_{n-1}+C_nP_{n-2},P_{n+2}+B_{n+2}P_{n+1}+C_{n+2}P_n)_0\\
=&C_{n+2}h_n^2,
\end{split}
\]
which gives~\eqref{DG_mrepr_h1}.
\end{proof}

The next step will be to establish a direct connection between the matrices   $J_{mon}$ and  $J_{mon}^{**}$
associated with the monic orthogonal polynomial sequences $\{P_n\}_{n=0}^{\infty}$, $\{P_n^{**}\}_{n=0}^{\infty}$,
respectively.\\
Let
\begin{equation}
 L_{mon}= \begin{pmatrix}
1 & 0&       &&\\
B_1 & 1   & 0 &&\\
C_2  &B_2   & 1 &0&\\
   0  & C_3  &B_3 &1 &\ddots\\
&&       &\ddots &\ddots\\
\end{pmatrix}
\end{equation}
be an infinite matrix such that   $ P^{**} = L_{mon} P$, where  $P^{**}= (P_{0}^{**}, P_{1}^{**}, \dots )^{\top}$  and
$P= (P_{0}, P_{1}, \dots )^{\top}.$ At the same time, from the equality
\[
[ t^{2}P_{n}(t),P_{m}^{**}(t)]_{2}= [P_{n}(t),P_{m}^{**}(t)]_{0}= 0, \quad m= 0, 1, \dots , n-1,
\]
one concludes that
\[
t^{2}P_{n}(t) = P_{n+2}^{**}(t)+ D_{n+1}P_{n+1}^{**}(t) + E_{n+1}P_{n}^{**}(t),\quad E_{n+1}\neq 0,\quad
n= 0, 1, \dots.
\]
The above connection formula reads in a matrix form as $t^{2} P = U_{mon}P^{**} $, where

\begin{equation}
U_{mon} ^{\top} = \begin{pmatrix}
E_1 & 0 &&\\
D_1 & E_2 & 0 &&\\
1  & D_2  & E_3 & 0 &&\\
0  & 1 &   D_3 &E_4&0 &&\\
0  & 0 &   1&  D_4 & E_5&\ddots &&\\
&&     &\ddots &\ddots&\ddots\\
\end{pmatrix}
\end{equation}

Thus, we can state the following.
\begin{theorem} We have that
 \begin{equation}
 J_{mon}^{2} = U_{mon} L_{mon},
\end{equation}
 \begin{equation}
 J_{mon}^{**}=  L_{mon } U_{mon}.
\end{equation}
\end{theorem}

\begin{proof}

Notice that
\begin{equation}
  t^{2} P=   U_{mon} P^{**}= U_{mon} L_{mon} P.
\end{equation}  Thus, one sees that
\begin{equation}
 J_{mon}^{2} = U_{mon} L_{mon}.
\end{equation}
On the other hand, we have
\begin{equation}
   t^{2} P^{**} = L_{mon} t^{2} P =  L_{mon} U_{mon} P^{**}.
\end{equation}
 As a consequence, one gets
\begin{equation}
  J_{mon}^{**}=  L_{mon } U_{mon}.
\end{equation}

\end{proof}
Notice that this is the analogue for pentadiagonal matrices of the Darboux transformation with parameter
considered in Section 3. Moreover, the structure of the matrix representing the multiplication operator
by $t^{2}$ with respect to the orthonormal polynomial basis associated with the inner product $[\cdot , \cdot]_{2}$
is stated in terms of the corresponding UL factorization of the pentadiagonal matrix $J^{2}$.\\

{\Large\bf Acknowledgments}

The research of MD is supported by the European Research Council under the European Union  Seventh
Framework Programme (FP7/2007-2013)/ERC grant agreement no. 259173. The research of FM has been supported by Direcci\'on General de Investigaci\'on, Ministerio de Econom\'ia y Competitividad of Spain, grant MTM2012-36732-C03-01.

\end{document}